\documentclass[%
	,a4paper 
	,11pt 
	,english
	]{amsart}

\usepackage[english]{babel}
\usepackage[parfill, indent]{parskip} 
\usepackage{ragged2e}
\usepackage{xcolor}

\let\ORGvarphi=\varphi
\let\varphi=\phi
\let\phi=\ORGvarphi

\let\ORGvarepsilon=\varepsilon
\let\varepsilon=\epsilon
\let\epsilon=\ORGvarepsilon

\usepackage{mathtools}
\DeclarePairedDelimiterX{\norm}[1]{\lVert}{\rVert}{\ifblank{#1}{\:\cdot\:}{#1}}
\DeclarePairedDelimiterX{\abs}[1]{\lvert}{\rvert}{\ifblank{#1}{\:\cdot\:}{#1}}

\makeatletter
\gdef\resetMathstrut@{%
 \setbox\z@\hbox{%
   \mathchardef\@tempa\mathcode`\!\relax%
   \def\@tempb##1"##2##3{\the\textfont"##3\char"}%
   \expandafter\@tempb\meaning\@tempa \relax
 }%
\ht\Mathstrutbox@\ht\z@ \dp\Mathstrutbox@\dp\z@}
\makeatother
\begingroup
 \catcode`(\active \xdef({\left\string(}
 \catcode`)\active \xdef){\right\string)}
 \catcode`[\active \xdef[{\left\string[}
 \catcode`]\active \xdef]{\right\string]}
\endgroup
\mathcode`(="8000 \mathcode`)="8000
\mathcode`[="8000 \mathcode`]="8000
\DeclareRobustCommand{\{}{\ifmmode\left\lbrace\else\textbraceleft\fi}
\DeclareRobustCommand{\}}{\ifmmode\right\rbrace\else\textbraceright\fi}

\usepackage{bbm}			        
\newcommand{\1}{\mathbbm{1}}
\newcommand{\CA}{C$^{\star}$}	    
\newcommand{\BA}{\mathfrak{A}}		
\newcommand{\LH}{\mathcal{L}(H)}
\newcommand{\Kp}[1]{\mathcal{K}(#1)}
\newcommand{\Sp}[1]{\mathop{}\mathrm{Sp}(#1)\,} 
\newcommand{\supp}[1]{\operatorname{s}(#1)}
\newcommand{\N}{\mathbb{N}}		
\newcommand{\Z}{\mathbb{Z}}		
\newcommand{\C}{\mathbb{C}}		
\newcommand{\T}{\mathbb{T}}     

\theoremstyle{plain}								
\newtheorem{theorem}{Theorem}[section]
\newtheorem{proposition}[theorem]{Proposition}			
\newtheorem{lemma}[theorem]{Lemma}				
\newtheorem{corollary}[theorem]{Corollary}

\theoremstyle{definition}						
	
\newtheorem{remark}[theorem]{Remark}
\newtheorem{example}[theorem]{Example}
\newtheorem{conjecture}[theorem]{Conjecture}
\newtheorem{open_problem}[theorem]{Open Problem}
\numberwithin{equation}{section}

\usepackage[inline,shortlabels]{enumitem}

\usepackage{xspace}
	\newcommand{\eg}{e.g.\xspace}
	\newcommand{\ie}{i.e.\xspace}

\title[Positivity of inverse operators on \CA-algebras]{A note on the positivity of inverse operators acting on \CA-algebras} 
\date{\today}

\author{Jochen Glück}
\address{%
Fakultät für Mathematik und Naturwissenschaften\\
Bergische Universität Wuppertal \\
	Gaußstr. 20\\
42119 Wuppertal (Germany)
	}
\email{glueck@uni-wuppertal.de}

\author{Ulrich Groh}
\address{%
Mathematisches Institut \\
Universität Tübingen \\
Auf der Morgenstelle 10 \\
72076 Tübingen (Germany)
}
\email{ulgr@math.uni-tuebingen.de}

\begin{document}

\begin{abstract}
    For a positive and invertible linear operator $T$ acting on a \CA-algebra, we give necessary and sufficient criteria for the inverse operator $T^{-1}$ to be positive, too.
    Moreover, a simple counterexample shows that $T^{-1}$ need not be positive even if $T$ is unital and its spectrum is contained in the unit circle.
\end{abstract}
\maketitle
\thispagestyle{empty}
\section{Introduction}
The following conjecture was recently formulated by Fidaleo, Ottomani, and Rossi in \cite[Conjecture 5.5]{fidaleo:2022}.

\begin{conjecture}
    \label{conj:automorphism}
    Let $T$ be a completely positive linear map with $T\1 = \1$ on a unital \CA-algebra $\BA$. 
    If the spectrum $\Sp{T}$ of $T$ is contained in the unit circle $ \T = \{ \lambda \in \C \colon \abs{ \lambda } = 1 \} $, is $T$ a \CA-automorphism?
\end{conjecture}

It follows for instance from \cite[Theorem~2.10]{bhat:2023} that the conjecture holds in finite dimensions. 
The present note serves the following purposes.

\begin{enumerate}[(i)]
    \item 
    We show by a simple counterexample that Conjecture~\ref{conj:automorphism} does not hold in general on infinite-dimensional \CA-algebras, not even in the commutative case (Example~\ref{exa:perturbed-shift}).

    \item 
    We give a short and very simple proof of Conjecture~\ref{conj:automorphism} in the finite-dimensional case which is self-contained except for the usage of elementary spectral theory. 

    \item 
    We observe that the conjecture is true if the assumption $\Sp{T} \subseteq \T$ is replaced by the stronger assumption that $T$ is a bijective isometry (Proposition~\ref{prop:isometry-unital}). 
    More interestingly, this remains true without the assumption that $T\1 = \1$ (Theorem~\ref{thm:isometry}).

    \item 
    We show that the conjecture is true if $T$ is doubly power bounded (which is stronger than $\Sp{T} \subseteq \T$ but weaker than $T$ being a bijective isometry) in the case where $\BA = \LH$ and $T$ has a pre-adjoint $T_{*}$ with a strictly positive (sub)fixed vector in the space of trace class operators (Theorem~\ref{thm:trace-class}). 
    More generally, an analogous result holds for operators acting on atomic von Neumann algebras (Theorem~\ref{thm:atomic-class}).
\end{enumerate}

The proof of Theorems~\ref{thm:trace-class} and~\ref{thm:atomic-class} is based on the so-called \emph{Jacobs-de Leeuw-Glicksberg decomposition} from operator theory, which we briefly recall in Section~\ref{sec:jdlg}. 
While our results in Section~\ref{sec:ap-operators-on-c-star-algebras} are special cases of those in Section~\ref{sec:atomic}, we prefer to spell them out explicitly and with different proofs which we believe can be insightful for readers who are mainly interested in the situation on $\LH$.

For the general theory of operator algebras we refer to Blackadar \cite{blackadar:2006} and for the theory of positive operators on operator algebras to St{\o}rmer \cite{stoermer-in:1973, stoermer:2013}. 
In Section~\ref{sec:positive-operators} we recall some terminology and a few simple results about positive operators acting on \CA-algebras (inclu\-ding some proofs) to provide additional context for the results in the latter sections.

\section{Positive Operators on \CA-algebras}
    \label{sec:positive-operators}

We are interested in the following three classes of bounded linear operators acting on a \CA-algebra $\BA$:

\begin{enumerate}[(i)]
    \item
    the positive operators, \ie operators that map the positive cone $\BA_{+}$ into $\BA_{+}$; 

    \item
    the Schwarz operators, \ie the operators satisfying
    \[
        T(x)^{*} T(x) \leq \norm{T} T(x^{*} x) \quad \text{for all } x \in \BA; 
    \]
    
    \item
    and the $n$-positive operators for $n \geq 2$, including the completely positive operators.
\end{enumerate}

It is well-known that (iii) implies (ii) and (ii) implies (i). 
If (and only if) the \CA-algebra is commutative, then positivity implies complete positivity and hence all three properties (i)--(iii) are equivalent (see \cite[Theorem~1.2.5]{stoermer:2013} for the details).
A positive operator $T$ on a \CA-algebra $\BA$ satisfies the \emph{Kadison inequality}
\begin{align}
    \label{eq:kadison}
    T(a)^{2} \leq \norm{ T } \, T( a^{2} ) \quad \text{for all } a \in \BA_{ h },
\end{align}
where $\BA_{h}$ denotes the self-adjoint part of $\BA$; see \cite[Theorem~1.3.1(iii)]{stoermer:2013}.

If $\BA$ is a \CA-algebra and $x$, $y \in \BA_{h}$, then
\[
    x \circ y \coloneqq \frac{1}{2} (x y + y x)
\]
is a  (non associative) product on $\BA_{ h }$, called the \emph{Jordan product}. 
A bounded linear operator $T \colon  \BA \to \BA$ that satisfies $T(x \circ y) = (Tx) \circ (Ty)$ for all $x$, $y  \in \BA_{h}$ is called a \emph{Jordan homomorphism}. 
The operator $T$ is called a \emph{Jordan automorphism} if $T$ is a bijective Jordan homomorphism (note that this implies that $T^{-1}$ is a Jordan homomorphism, too).
Every Jordan homomorphism is positive since the positive elements in $\BA_{h}$ are precisely those elements that have a square root in $\BA_{h}$.
Hence, if $T$ is a Jordan automorphism, then both $T$ and $T^{-1}$ are positive. 
In the following proposition we recall a partial converse to this observation.

\begin{proposition}
    \label{prop:posop-jordan}
    Let $T \colon  \BA \to \BA$ be a bijective linear operator on a \CA-algebra $\BA$. 
    If both $T$ and $T^{-1}$ are contractive and positive, then $T$ is a Jordan automorphism.
\end{proposition}

\begin{proof} 
    We first show that 
    $
        T( a )^{ 2 } = T( a^{ 2 } ) 
    $
    for all $a \in \BA_{h}$.
    Indeed, let $a \in \BA_{h}$. 
    Since $\norm{T} \leq 1 $ and $\norm{T^{-1}} \leq 1$, employing the Kadison inequality~\eqref{eq:kadison} twice gives 
    \begin{align*}
        a^{ 2 } 
        & = 
        ( T^{ -1 }( Ta ) ) \cdot ( T^{ -1 }( Ta ) ) 
        \\
        & \leq 
        T^{ -1 }( ( Ta )^{ 2 } )  
        \leq 
        ( T^{ -1 } \circ T ) a^{ 2 } = a^{ 2 } 
    \end{align*}
    and therefore $T ( a^{ 2 } ) = T( a )^{ 2 } $. 
    This implies the claim since 
    \[
        2 x \circ y 
        = 
        ( x + y )^{ 2 } - x^{ 2 } - y^{ 2 }
    \]
    for all $x$ ,$y \in \BA_{h}$.
\end{proof}

If $T$ and $T^{-1}$ are even Schwarz operators, then the previous proposition can be strengthened.

\begin{proposition}
    \label{prop:schwarz-star-automorphism}
    Let $T \colon  \BA \to \BA$ be a bijective linear operator on a \CA-algebra $\BA$. 
    If both $T$ and $T^{-1}$ are contractive Schwarz operators, then 
    \[
        T (x)^{*}T(y) = T(x^{*} y) 
    \]
    for all $x$, $y  \in \BA$, \ie,  $T$ is a \CA-automorphism.
\end{proposition}

\begin{proof}
    For $T$ being Schwarz operator, we obtain for all $ x \in \BA $ 
    \begin{align*}
        x^* x 
        & = 
        ( T^{ -1 } ( T x ) )^{ * } \cdot ( T^{ -1 } ( T x ) ) 
        \\
        & \leq 
        T^{ -1 } \big( ( Tx )^{ * } ( Tx ) \big) 
        \leq 
        (T^{ -1 } \circ T ) ( x^{ * } x ) 
	= 
        x^* x,
    \end{align*}
    and hence
    $
        T(x^{*} x) = T (x)^{*}T(x)
    $.
    Using the polarization identity
    \[
        4 y^{*}x = \sum_{k=0}^{3} \text{i}^{k} ( x + \text{i}^{k} y)^{*}( x + \text{i}^{k} y)
    \]
    we thus obtain $T (y^{*} x ) = T (y)^{*} T(x) $ for all $x$, $y \in \BA$, as claimed.
\end{proof}

We conclude this section by recalling the following classical result; 
a proof can be found in \cite[Proposition~II.6.9.4]{blackadar:2006}.

\begin{proposition}
    \label{prop:positivity-by-one}
    Let $T \colon  \BA \to \BA$ be a bounded unital linear map on a unital \CA-algebra $\BA$. 
    Then $T$ is positive if and only if $T$ is a contraction.
\end{proposition}

\section{A counterexample}
    \label{sec:example}

A bounded linear operator $T$ on a Banach space $E$ is called \emph{doubly power bounded} if it is bijective and satisfies
\[
    \sup_{n \in \Z } \norm{ T^{n} } < \infty.
\]
For such operators we always have $ \Sp{T} \subseteq  \T $.
In \cite{emelyanov:2004}, Emel'yanov gave an example of a doubly power bounded and positive operator $T$ on an $L^{1}$-space such that the inverse operator $T^{-1}$ is not positive. 
In \cite{alpay:2006} similar examples were shown to exist in fact on every infinite-dimensional $L^{1}$-space.
By taking dual operators, one thus obtains a doubly power bounded positive operator on a commutative von Neumann algebra such that the inverse operator is not positive. 
Yet, this result does not give a complete answer to Conjecture~\ref{conj:automorphism} since the operators on $L^{1}$ that are constructed in \cite{emelyanov:2004} and \cite{alpay:2006} are not norm preserving on the positive cone and hence, their dual operators are not unital.   

The following example is an adaptation of the construction in \cite{emelyanov:2004} and shows that the answer to Conjecture~\ref{conj:automorphism} is negative, even on commutative \CA-algebras.

\begin{example}
    \label{exa:perturbed-shift}
    Let $ \BA $ be the commutative \CA-Algebra
    \[
        \BA  = \ell^{\infty}(\Z) \times \C
    \]
    which is unital with $ \1 = ( e , 1 ) $, where $ e $ is the unit in $ \ell^{\infty}(\Z) $.
    We construct a unital and (completely) positive operator $T$ on $\BA$ that is doubly power bounded with $T^{-1}$ not positive. 
    Due to the double power boundedness, one has $\Sp{T} \subseteq \T$.

    To this end, let $R$ be the right shift on $ \ell^{\infty}(\Z) $ and $M$ the multiplication operator on $ \ell^{\infty}(\Z) $ given by
    \[
        Mx(j) = 
        \begin{cases}
            x(j) \quad & \text{if } j \neq 0 , \\
            \frac{1}{2} x(0) \quad & \text{if } j = 0 
        \end{cases}
    \]
    for all $x = ( x(j) ) \in \ell^{\infty}(\Z)$.
    Define the operator $T \colon \BA \to \BA $ by
    \[
        T ( x, \alpha ) = ( R M x  + \frac{1}{2} \alpha e_{1} , \alpha )   
    \]
    for all $ ( x , \alpha ) \in  \BA  $, where $ e_{1} \in \ell^\infty(\Z)$ is the canonical unit vector which is $1$ at index $1$ and $0$ at every other index.
    
    Then $T$ is a positive (hence completely positive) operator satisfying $T\1 = \1$ and $\norm{T^n} \le 1$ for all integers $n \ge 0$.
    The inverse of $T$ is given by
    \[
        T^{-1} ( x , \alpha ) = (M^{-1} L x - \alpha e_{0} , \alpha ) , 
    \]
    for all $(x, \alpha) \in \BA$, where $ L $ the left shift on $ \ell^{\infty}(\Z) $ and $e_{0} \in \ell^\infty(\Z)$ is the canonical unit vector that is $1$ at the index $0$. 

    So $T^{-1}$ is not positive since it maps $(0,1)$ to $(-e_{0},1)$.
    By using the formula for $T^{-1}$ one can easily check that $\norm{T^{-n}} \le 3$ for each integer $n \ge 1$ hence $T$ is doubly power bounded.
\end{example}

\begin{remark}
    Let $T$ be the operator constructed in Example~\ref{exa:perturbed-shift}. 
    Then every number $\lambda \in \T$ is an eigenvalue of $T$ which is geometrically simple if $\lambda \neq 1$ and has geometric multiplicity $2$ if $\lambda = 1$.

    More precisely, the fixed space $\ker(1-T)$ is spanned by the unit $ \1 $ and the vector $(x, 0)$ for $x \in \ell^\infty(\Z)$ given by 
    \[
        x(j) = 
        \begin{cases}
            2, \quad & \text{if } j \leq 0,   \\
            1, \quad & \text{if } j \geq 1.
        \end{cases}
    \]
    For every $1 \neq \lambda \in \T$ the eigenspace $\ker(\lambda - T)$ is spanned by the eigenvector $(x, 0) $ where $x \in \ell^\infty(\Z)$ is given by 
    \begin{align*}
        x(j) = 
        \begin{cases}
            2 \lambda^{-j} \quad & \text{if } j \le 0, \\
              \lambda^{-j} \quad & \text{if } j \ge 1.
        \end{cases}
    \end{align*}
\end{remark}

\section{Positive inverses in finite dimensions}
    \label{sec:finite-dime} 

In this section we give a short and self-contained proof that Conjecture~\ref{conj:automorphism} holds on finite-dimensional \CA-algebras. 
More precisely, we show the following result.

\begin{theorem}
    \label{thm:automorphism-finite-dim}
    Let $\BA$ be a finite-dimensional \CA-algebra and let $T \colon  \BA \to \BA$ be a unital linear map with spectrum $\Sp{T} \subseteq \T$. 
    \begin{enumerate}[\upshape(i)]
        \item  
        If $T$ is positive, then so is $T^{-1}$ and $T$ is a Jordan automorphism.

        \item 
        If $T$ is a Schwarz operator, then $T$ is even a \CA-automorphism.

        \item 
        In particular, if $T$ is $n$-positive for some $n \ge 2$, then $T$ is a \CA-automorphism.
    \end{enumerate}
\end{theorem}

Before the proof a few remarks are in order.

\begin{remark}
    \begin{enumerate}[(i)]
        \item
        On a commutative \CA-algebra every positive operator is already completely positive, hence a Schwarz operator. 
        Example~\ref{exa:perturbed-shift} shows that Theorem~\ref{thm:automorphism-finite-dim} is no longer true if the \CA-algebra is infinite-dimensional.
        
        \item
        If $\BA$ is a finite-dimensional commutative \CA-algebra and if $T $ satisfies the assumptions of Theorem~\ref{thm:automorphism-finite-dim}(i), then $T $ is a row stochastic positive matrix. 
        It follows that $T^{-1}$ is also positive and thus $T$ is a permutation matrix. 
        This is known and can be found in~\cite[Theorem~I.4.5]{schaefer:bl}.
        See also \cite[Theorem~1.1.15]{emelyanov:2007} where almost the same result (for power bounded rather than row stochastic matrices) is proved by means of the Jacobs-de Leeuw-Glicksberg decomposition -- a method that we discuss in Section~\ref{sec:jdlg} below.
    \end{enumerate}
\end{remark}

Our proof of Theorem~\ref{thm:automorphism-finite-dim} is based on the following two simple lemmas. 
The first one is a special case of well-known results in ergodic theory, see for instance \cite[Proposition~3.12(d)]{efhn:2016}. 
However, we include a proof not needing ergodic theory.

Recall that \emph{topological group} is a group endowed with a topology which makes the group operation and the inverse operation continuous.
\begin{lemma}
    \label{lem:group-rotations-are-recurrent} 
    Let $G$ be a topological group with neutral element $1$ and assume that $G$ is compact and metrizable. 
    For every $g\in G$ there exists a sequence of integers $ 1 \leq n_{k} \to \infty$  such that $ g^{n_{k}} \to 1$.
\end{lemma}
\begin{proof}
    Fix $g \in G$.
    We call a set $S$ \emph{left invariant under $g$} if $gS \subseteq S$. 
    It follows from the compactness of $G$ and from Zorn's lemma that there exists a set $M \subseteq G$ which is minimal (with respect to set inclusion) among all non-empty, closed subsets of $G$ that are left invariant under $g$. 

    Fix a point $h \in M$.
    The minimality of $M$ implies that, for each integer $n_{0} \ge 0$, the orbit tail 
    \[
        \{g^{n} h \colon n \ge n_0\}
    \]
    is dense in $M$.
    In particular, each of those tails contains $h$ in its closure. 
    Hence, we can iteratively construct a strictly increasing sequence of integers $1 \le n_{k} $ such that 
    \[
        \operatorname{d}(g^{n_{k}}h, h) \le \frac{1}{k} \quad \text{for each $k \ge 1$.} 
    \]
    So we have $ \lim_{k}g^{n_{k}}h = h $ and by multiplying from the right with $h^{-1}$ we obtain $ \lim_{k} g^{n_{k}} = 1$, as claimed.
\end{proof}

We now apply this lemma to power bounded operators on the matrix algebra $M_{n}(\C)$.

\begin{lemma}
    \label{lem:powers-unit-circle}
    Let $m \ge 1$ be an integer. 
    \begin{enumerate}[\upshape(i)]
        \item  
        Let $\lambda_{1}, \dots, \lambda_{m} \in \T$. 
        Then there exists a sequence of integers $n_{k} \to \infty$ such that $\lambda_j^{n_{k}} \to 1$ as $k \to \infty$ for each $j \in \{1, \dots, m\}$.

        \item 
        Let $T \in M_{n}(\C)$ be power bounded and assume that $\Sp{T} \subseteq \T$. 
        Then there exists a sequence of integers $n_{k} \to \infty$ such that $T^{n_{k}} \to \operatorname{id}$ as $k \to \infty$, where $\operatorname{id}$ denotes the identity operator on $M_{n}(\C)$.

        \item 
        Let $T \in M_{n}(\C)$ be power bounded and let $P$ denote the spectral projection of $T$ that belongs to the part $\Sp{T} \cap \T$ of the spectrum. 
        Then there exists a sequence of integers $n_{k} \to \infty$ such that $T^{n_{k}} \to P$ as $k \to \infty$.
    \end{enumerate}
\end{lemma}

\begin{proof}
    (i) 
    Consider the element $\lambda = (\lambda_1, \dots, \lambda_m)$ in the compact group $\T^{m}$ (endowed with componentwise multiplication). 
    It follows from Lemma~\ref{lem:group-rotations-are-recurrent} that there exists a sequence of integers $n_{k} \to \infty$ such that $\lambda^{n_{k}}  \to 1$ as $k \to \infty$.
    
    (ii) 
    Since $T$ is power bounded and all its eigenvalues are located on $\T$, the Jordan normal form of $T$ only contains Jordan blocks of size $1$, \ie, $T$ is diagonalizable. 
    Hence, the claim readily follows from~(i).

    (iii) 
    This follows from (ii) by restricting $T$ to the image of $P$.
\end{proof}

The observation in Lemma~\ref{lem:powers-unit-circle}(i) is used implicitly in some proofs about operators acting on finite-dimensional \CA-algebras, for instance in \cite[Proposition~3.1]{fidaleo:2023} and \cite[Theorem~2.1]{kuperberg:2003}.

\begin{proof}[Proof of Theorem~\ref{thm:automorphism-finite-dim}]
    (i) 
    Since $T$ is unital and positive, it is contractive by Proposition~\ref{prop:positivity-by-one} and hence power bounded. 
    As $\Sp{T} \subseteq \T$, Lemma~\ref{lem:powers-unit-circle}(ii) shows that there exists a sequence of integers $1 \le n_{k} \to \infty$ such that $T^{n_{k}} \to \operatorname{id}$. 
    Thus, $T^{n_{k}-1} \to T^{-1}$, which shows that $T^{-1}$ is positive, too. 
    
    Since $T^{-1}$ is unital, too, it is also contractive by Proposition~\ref{prop:positivity-by-one}.  
    Thus, Proposition~\ref{prop:posop-jordan} can be applied and gives that $T$ is a Jordan automorphism.

    (ii) 
    As shown in~(i), one has $T^{-1} = \lim_{k \to \infty} T^{n_{k}-1}$, so $T^{-1}$ is Schwarz.
    Since $T$ and $T^{-1}$ are contractive, it follows from Proposition~\ref{prop:schwarz-star-automorphism} that $T$ is a \CA-automorphism.

    (iii) 
    This is an immediate consequence of~(ii), since every $2$-positive operator is a Schwarz operator.
\end{proof}

\section{Positive inverses in infinite dimensions I: the isometric case}
    \label{sec:positive-isometries}

If $T$ is a positive and unital operator on an infinite-dimensional unital \CA-algebra, the assumption $\Sp{T} \subseteq \T$ does not suffice to ensure that $T^{-1}$ is positive as seen in Example~\ref{exa:perturbed-shift}. 
Throughout the rest of this paper we discuss stronger assumptions giving positivity of $T^{-1}$. 
The first condition is that $T$ is a bijective isometry. 
This is a strong assumption and the proof is a simple consequence of Proposition~\ref{prop:positivity-by-one}.

\begin{proposition}
    \label{prop:isometry-unital}
    Let $T \colon  \BA \to \BA$ be a positive and unital operator on a unital \CA-algebra $\BA$.
    The following are equivalent:
    \begin{enumerate}[\upshape(a)]
         \item 
         One has $0 \not\in \Sp{T}$ and the inverse operator $T^{-1}$ is positive.

         \item 
         The operator $T$ is a bijective isometry.
    \end{enumerate}
\end{proposition}

\begin{proof}
    In both cases the operator $T$ is a bijection and $0 \not\in \Sp{T}$

    (a) $\implies$ (b)  
    Since $T$ and $T^{-1}$ are both positive and unital, they are both contractive according to Proposition~\ref{prop:positivity-by-one}.

    (b) $\implies$ (a)  
    Since $T$ is isometric, $T^{-1}$ is contractive and satisfies $T^{-1}\1 = \1$.
    Proposition~\ref{prop:positivity-by-one} implies that $T^{-1}$ is positive.
\end{proof}

The next result is slightly more involved and shows that the implication (b) $\implies$ (a) in Proposition~\ref{prop:isometry-unital} remains true if $T$ is not assumed to be unital (and even if the underlying \CA-algebra is not unital).
An analogous result is known to hold for positive operators an Banach lattices, see \cite[Theorem~2.2.16]{emelyanov:2007}. 

\begin{theorem}
    \label{thm:isometry}
    Each positive surjective linear isometry on a \CA-algebra has a positive inverse.
\end{theorem}

\begin{proof}
    Let $T$ be a positive and surjective linear isometry on the \CA-algebra $ \BA $.
    We have to show that $T(\BA_{+}) = \BA_{+} $.
    
    In the first step we show $T(\BA_{h}) = \BA_{h} $.
    Indeed, if $  b \in \BA_{h} $, then there exists $ x \in \BA $ such that $Tx = b $.
    We can decompose $ x $ into its real part $x_{1}$ and its imaginary part $ x_{2} $, \ie,
    \[
        x = x_{1} + \text{i}\,x_{2} 
        \quad \text{, hence} \quad 
        b = T(x_{1}) + \text{i}\,T( x_{2} ).
    \]
    Since $T $ is positive, $Tx_{1} $ and $Tx_{2} $ are self-adjoint, so $T(x_{2}) = 0$ and $  b = T(x_{1}) $.
    
    If $y \in \BA_{+}$ and $\norm{y}=1$, we find, as we have just shown, $ x \in \BA_{h} $, $\norm{x} = 1$, such that $T(x) = y$.
    We now decompose $ x $ into its positive part $ x^{+} $ and it's negative part $ x^{-} $, \ie, 
    \[
        x = x^{+} - x^{-},  
        \quad 
        x^{+}x^{-} = x^{-} x^{+} = 0,  
        \quad 
        \norm{x} = \max \{\norm{x^{+}}, \norm{x^{-}} \}.
    \]
    We have to show that $x^- = 0$.
    The positive elements
    \[
        y_{1} = T x^{+}, 
        \quad 
        y_{2} = T x^{-}
    \]
    of $\BA$ satisfy $0 \le y = y_{1} - y_{2} $.
    
    We now prove by induction that, for each $ n \in \N $,
    \begin{equation}
        \label{eq:norm-estimate}
        \norm{  x^{+} - n x^{-} } \leq 1 . 
    \end{equation}
    The case $ n=1 $ is known, so assume now that~\eqref{eq:norm-estimate} has been shown for a fixed $n \in \N$.
    One has 
    \[
        -(y_{1} + n y_{2}) \leq y_{1} - (n+1) y_{2} \leq y_{1} + n y_{2} .
    \]
    
    Since $ y_{1} $ and $ y_{2} $ are self-adjoint we thus obtain the norm estimate
    \begin{align*}
        \norm{ y_{1} - (n+1) y_{2} } & \leq \norm{ y_{1} + n y_{2} } = \norm{ T( x^{+} + n x^{-}) }  \\
        = \norm{ x^{+} + n x^{-}} &= \norm{ x^{+} -n x^{-}} = 1 .
    \end{align*}
    Again, since $T$ is isometric, it follows that 
    \[
        \norm{ x^{+} - (n+1)x^{-} } \leq 1, 
    \]
    proving~\eqref{eq:norm-estimate}. 
    Now we divide~\eqref{eq:norm-estimate} by $n$ and let $n \to \infty$ to obtain $x^- = 0$.
\end{proof}

\begin{remark}
    \begin{enumerate}[(i)]
        \item 
        If $T$ is a surjective isometry on a Banach space, then $\Sp{T} \subseteq \T$ if $T$ is surjective.
        If $T$ is not surjective, then its spectrum is equal the set $ \mathbb{D} = \{ \lambda \in \C \colon \abs{\lambda} \leq 1 \}$ (see \eg \cite[Theorem~3.3]{fidaleo:2022}). 
        Since injective \CA-homomorphism are isometric, the theorem gives a spectral characterization of \CA-automorphism.
        
        \item 
        Isometric operators between unital \CA-algebras are discussed in \cite{kadison:1951} and more can be found in \cite{stoermer-in:1973}.
    
        \item 
        An operator $T $ on a \CA-algebra $ \BA $ is called $ n $-isometric if its canonical extension $T^{(n)} $ to $ M_{n}(\BA) $ is isometric.
        Assume that $T$ is $n$-positive, $n$-isometric and surjective.
        Then we can apply Theorem~\ref{thm:isometry} to the extension $T^{(n)} $ and obtain that $T^{-1} $ is $n$-positive. 
        Hence, $T$ is a \CA-automorphism by Proposition~\ref{prop:schwarz-star-automorphism}.
        As a special case we obtain \cite[Corollary~1.3.10]{blecher:2004}: Every completely isometric unital surjection on an \CA-algebra is already a \CA-automorphism.
    \end{enumerate}
\end{remark}

\section{Intermezzo: the Jacobs-de Leeuw-Glicksberg decomposition}
    \label{sec:jdlg}

The arguments used to prove the finite-dimensional Theorem~\ref{thm:automorphism-finite-dim} relied on the fact that a power bounded matrix acts as a multiplication on its eigenspaces and on the fact that one can find a sequence $1 \le n_{k} \to \infty$ such that $\lambda^{n_{k}} \to 1$ simultaneously for all $\lambda$ from a finite subset of the unit circle $\T$ (Lemma~\ref{lem:powers-unit-circle}). 
These arguments cannot be directly generalized to infinite-dimensional spaces (although Lemma~\ref{lem:powers-unit-circle}(a) can be generalized to obtain $\lambda^{n_{k}} \to 1$ simultaneously for all $\lambda \in \T$ if one allows $(n_{k})$ to be a subnet rather than a subsequenes of the natural numbers -- the net is needed due to the fact that the product topology on $\T^{\T}$ is not metrizable).

However, there is an infinite-dimensional substitute for the arguments from Section~\ref{sec:finite-dime}. 
This is the so-called \emph{Jacobs-de Leeuw-Glicksberg decomposition}, which we now formulate in a version that is appropriate for our application to positive operators on \CA-algebras. 
We refer for instance to \cite[Chapter~16]{efhn:2016} or \cite[Section~2.4]{krengel:1985} for the details of this theory. 

To make the decomposition work, a compactness assumption on the orbits of the operator is needed.
A bounded linear operator $T$ on a Banach space $E$ is called \emph{almost periodic} if the orbit $\{T^n x \colon  n \in \N_0\}$ is relatively norm compact in $E$ for each $x \in E$. 
Note that almost periodicity implies power boundedness by the uniform boundedness theorem.

\begin{theorem}
    \label{thm:jdlg}
    Let $T$ be a bounded linear operator on a complex Banach space $E$. 
    I{f\/} $T$ is almost periodic, then there exists a bounded linear projection $P$ on $E$ with the following properties. 
    \begin{enumerate}[\upshape (i)]
        \item 
        The projection $P$ commutes with $T$ and hence the kernel $\ker P$ and the range $PE$ are invariant under $T$. 

        \item 
        The restriction $T|_{PE}$ is bijective from $PE$ to $PE$. 

        \item 
        There exists a net of integers $1 \le n_{k} \to \infty$ such that $T^{n_{k}} \to P$ strongly.

        \item 
        The kernel of $P$ is given by $\ker P = \{x \in E \colon  T^n x \to 0 \text{ as $n \to \infty$}\}$.

        \item 
        The range $E_{\text{rev}} := PE$ of $P$ is the closed linear span of the eigenspaces $\ker(\lambda-T)$, where $\lambda$ runs through all eigenvalues of $T$ with modulus $1$.
    \end{enumerate}
\end{theorem}

The subspace $E_{\text{rev}} = PE$ in part~(v) above is called the \emph{reversible part} of the space $E$.

\begin{proof}[Proof of Theorem~\ref{thm:jdlg}]
    All these properties can be found in the above mentioned references \cite[Chapter~16]{efhn:2016} and \cite[Section~2.4]{krengel:1985}. 
    The only difference is that, in these references, the relative compactness of the orbits $\{ T^{n}x \colon  n \in \N_{0} \}$ is only assumed with respect to the weak topology and thus, the net $(T^{n_{k}})$ in~(iii) converges to $P$ in the weak operator topology only and in~(iv) the  subspace $\ker P$ consists of those $x$ for which $0$ is in the weak closure of the orbit $\{T^n x \colon n \in \N_{0}\}$. 

    The relative norm compactness of the orbits then readily gives~(iv) as stated in the theorem. 
    To get our version of~(iii) one only has to observe that the relative norm compactness of all orbits $\{T^n x \colon n \in \N_{0}\}$ is equivalent to relative compactness of the set $\{T^n \colon n \in \N_0\}$ in the strong operator topology, see for instance \cite[Corollary~A.5]{engel:2000}.
\end{proof}

Further details on the Jacobs-de Leeuw-Glicksberg decomposition under the assumption that the orbits are relatively compact in norm can be found in \cite[Subsection~V.2(c)]{engel:2000}.

\begin{remark}
    For a bounded linear operator $T$ on a finite-dimensional complex Banach space $E$, almost periodicity of $T$ is equivalent to power boundedness. 
    Then the projection $P$ from Theorem~\ref{thm:jdlg} is the spectral projection that belongs to the part of $\Sp{T}$ on the unit circle.
    Hence, assertion~(iii) of the theorem is an infinite-dimensional version of Lemma~\ref{lem:powers-unit-circle}(iii).
\end{remark}

\section{Positive inverses in infinite dimensions II: the doubly power bounded case on $\LH$}
    \label{sec:ap-operators-on-c-star-algebras}

Let $\LH$ denote the \CA-algebra of all bounded linear operators on a complex Hilbert space $H$. 
By $\Kp{H}$ we denote the ideal of all compact operators and by $\mathcal{L}^1(H)$ the ideal of all trace class operators in $\LH$.

We use the Jacobs-de Leeuw-Glicksberg decomposition from Theorem~\ref{thm:jdlg} to prove the following result.
Related arguments are used in \cite{jdlg:2010, carbone:2020}

\begin{theorem}
    \label{thm:trace-class}
    Let $H$ be a complex Hilbert space and let $T \colon  \LH \to \LH$ be a positive and unital linear operator. 
    Assume that $T$ has a pre-adjoint $T_{*}$ acting on the space $\mathcal{L}^1(H)$ and assume that there exists an injective trace class operator $0 \le b \in \mathcal{L}^{1}(H)_{+}$ that satisfies $T_{*}b \le b$.
    Then the following are equivalent.
    \begin{enumerate}[\upshape(a)]
        \item 
        The inverse operator $T^{-1}$ exists and is positive.

        \item 
        The operator $T$ is a bijective isometry.
        
        \item 
        The operator $T$ is doubly power bounded.
    \end{enumerate}
\end{theorem}

The implication from (c) to (a) can be seen as a partial positive answer to Conjecture~\ref{conj:automorphism}: We replaced the assumption $\Sp{T} \subseteq \T$ by the stronger assumption that $T$ be doubly power bounded and obtain, under the assumptions in the theorem, that $T^{-1}$ is positive.

It is not clear to us whether this implication remains true if one replaces the double power boundedness of $T$ by the weaker assumption $\Sp{T} \subseteq \T$. 
Let us pose this explicitly as an open problem.

\begin{open_problem} 
    \label{probl:without-doubly-pb}
    Let $H$ be a complex Hilbert space and let $T \colon  \LH \to \LH$ be a positive and unital linear operator which has a pre-adjoint $T_{*}$ acting on $\mathcal{L}^1(H)$. 
    Assume that $T_{*} b = b$ for an injective operator $b \in \mathcal{L}^1(H)_{+}$. 
    If $\Sp{T} \subseteq \T$, does it follows that $T^{-1}$ is positive?
\end{open_problem}

We do not even know the answer to this problem in the commutative situation, see the discussion after Corollary~\ref{cor:commutative} below. 
For the proof of Theorem~\ref{thm:trace-class} we need a couple of auxiliary results.

\begin{lemma}
    \label{lem:order-interval-compact}
    Let $H$ be a complex Hilbert space and let $a$ be a positive operator in $\Kp{H}$.
    Then the order interval
    \[
        [0,a] = \{ c \in  \Kp{H}  \colon 0 \leq c \leq a \}
    \]
    is norm compact in $\Kp{H}$.
\end{lemma}

\begin{proof}
    Let $ (x_{\alpha} )_{\alpha \in A} $ be a net in $ [0,a] $ and choose a number $0 < \gamma < 1/2$. 
    According to \cite[Corollary~II.3.2.5]{blackadar:2006} there exists a norm bounded net $(y_{ \alpha })_{\alpha \in A}$ in $\LH_{+}$ such that 
    \[
        x_{ \alpha } = a^{ \gamma } y_{ \alpha } a^{ \gamma } 
    \]
    for each $\alpha \in A$.
    A subnet $(y_{\alpha_{\beta}})_{\beta \in B}$ of $(y_\alpha)_{\alpha \in A}$ converges to an operator $y \in \LH_{+}$ with respect to the weak operator topology. 
    To simplify the notation, we denote this subnet by $(y_{\alpha})_{\alpha \in A}$ from now on.
    
    Since $a^{\gamma}$ is compact, it follows that a $(a^{\gamma} y_{\alpha})_{\alpha \in A}$ converges to $a^{\gamma} y$ with respect to the strong operator topology. 
    Again due to the compactness of $a^\gamma$ this implies that 
    \[
    (a^\gamma y_{\alpha}a^\gamma)_{\alpha \in A} = (x_{\alpha})_{\alpha \in A}    
    \]
    converges with respect to the operator norm in $\LH$. 
    Since $\Kp{H}$ is norm closed in $\LH$ and $[0,a]$ is norm closed in $\Kp{H}$ we conclude that $[0,a]$ is indeed norm compact.
\end{proof}

\begin{lemma}
    \label{lem:ell-p-split}
    Let $p \in \left[1,\infty\right[$ and $0 \le x \in \ell^{p}$. 
    Then there exists $0 \le y \in \ell^{p}$ and $0 \le z \in c_{0}$ such that $x = yz$.
\end{lemma}

\begin{proof}
    It suffices to prove the lemma for $p=1$ since for other values of $p$ the claim then follows by considering the componentwise power $x^{p}$; so let $p=1$ and we also assume that all entries of $x$ are non-zero.

    We construct an increasing sequence of vectors $0 \le y^{(n)} \in \ell^{1}$ recursively as follows. 
    Set $y^{(1)} \coloneqq x$. 
    If $y^{(n)}$ has already been defined, choose an index $k_{n} \in \N$ such that the tail sum estimate 
    \[
        \sum_{k=k_{n}}^\infty y^{(n)}_{k_0} \le \frac{1}{2^{n}}
    \]
    holds. 
    We set 
    \begin{align*}
        y^{(n+1)}_{k} 
        \coloneqq 
        \begin{cases}
            y^{(n)}_{k}  \quad & \text{if } k < k_{n}, \\ 
            2y^{(n)}_{k} \quad & \text{if } k \ge k_{n},
        \end{cases}
    \end{align*}
    Then 
    \[
        \norm{y^{(n+1)} - y^{(n)}}_{1} \le \frac{1}{2^n}
    \]
    for each $n$.
    Therefore the increasing sequence $(y^{(n)})_{n \in \N}$ is Cauchy in $\ell^{1}$ and thus converges to a limit $y \in \ell^{1}$ satisfying $y \ge x$. 
    In particular, $y_{k} > 0$ for each index $k$. 
    Define $z_{k} \coloneqq x_{k}/y_{k}$ for each $k$. 

    For every $n \in \N$ and index $k$ with 
    $k \ge \max \{ k_{1}, \ldots , k_{n} \}$ one has 
    \[
        y_{k} \ge y^{(n+1)}_{k} \ge 2^n x_{k}
    \]
    and thus $\frac{1}{2^n} \ge z_{k}$. 
    So, indeed, $z \in c_0$.
\end{proof}

\begin{lemma}
    \label{lem:trace-class-order-intervals}
    Let $H$ be a complex Hilbert space and consider a positive trace class operator $b$ in $\mathcal{L}^{1}(H)$.
    \begin{enumerate}[\upshape (i)]
        \item 
        The order interval
        \[
            [0,b] \coloneqq \{a \in \mathcal{L}^{1}(H) \colon 0 \le a \le b \}
        \]
        is norm compact in $\mathcal{L}^{1}(H)$.

        \item 
        If $b$ is injective, then the linear span of $[0,b]$ is dense in $\mathcal{L}^{1}(H)$.
    \end{enumerate}
\end{lemma}

\begin{proof}
    (i)
    We start with the following preliminary observation.
    For every $a \in [0,b]$ and every $\gamma \in \left]0,1/2\right[$ there exists, again according to \cite[Corollary~II.3.2.5]{blackadar:2006}, an operator $c_\gamma \in \Kp{H}_{+}$ such that $a = b^\gamma c b^\gamma$ and $\norm{c} \le \norm{b^{1-2\gamma}} = \norm{b}^{1-2\gamma}$. 
    By using the compactness of closed balls in $\LH$ with respect to the weak operator topology we can thus find an operator $c \in \LH_{+}$ such that $a = b^{1/2} c b^{1/2}$ and $\norm{c} \le 1$. 
    (See also \cite[II.3.2.4]{blackadar:2006} for such an observation in the setting of von Neumann algebras).

    Now let $(a_{n})$ be a sequence in $[0,b]$. 
    According to our preliminary observation we can find a sequence $(c_{n}) \in \LH_{+}$ such that $a_{n} = b^{1/2} c_{n} b^{1/2}$ and $\norm{c_{n}} \le 1$ for each $n$. 
    The operator $b^{1/2}$ is Hilbert-Schmidt, so the sequence of its eigenvalues is in $\ell^{2}$ and we can thus apply Lemma~\ref{lem:ell-p-split} to this sequence. 
    Thus we obtain a Hilbert-Schmidt operator $y \ge 0$ and a compact operator $z \ge 0$ such that $b^{1/2} = yz = zy$ and hence, 
    $a_{n} = yz c_{n} zy$ for each index $n$. 
    
    The sequence $(zc_{n}z)$ is contained in the order interval $[0,z^2]$ in the space $\Kp{H}$.
    So according to Lemma~\ref{lem:order-interval-compact} it has a subsequence $(z c_{n_{k}} z)$ which norm converges to an element of $\Kp{H}$.
    Finally, observe that $z \mapsto yzy$ is a continuous mapping from $\Kp{H}_{+}$ to $\mathcal{L}^1(H)_{+}$ since $y$ is a Hilbert-Schmidt operator. 
    Hence, $(a_{n_{k}})$ converges in $\mathcal{L}^{1}(H)$. 
    Since the order interval $[0,b]$ is norm closed in $\mathcal{L}^1(H)$ we thus obtain the claimed compactness.

    (ii)
    It suffices to show that for every $h \in H$ the finite rank operator $h \otimes h$ given by $(h \otimes h) \xi = (\xi | h)h$ for all $\xi \in H$ is in the closure of the span of $[0,b]$; so fix $h \in H$. 

    Since the compact operator $b$ is injective, the space $H$ is separable. 
    If $H$ is finite-dimensional the claim is clear, so assume now that $H$ is infinite-dimensional.
    Then there exists an orthonormal basis $(\xi_{n})_{n \in \N}$ of $H$ consisting of eigenvectors of $b$ such that the corresponding eigenvalues $\lambda_{n}$ decrease to $0$ and note that they are all non-zero since $b$ is injective.

    For every $n \in \N$ let $P_{n}$ denote the orthogonal projection onto the span of $\{\xi_{1}, \dots, \xi_{n}\}$. 
    Then 
    \begin{align*}
        0 
        \le 
        P_{n}(h \otimes h)P_{n} 
        \le 
        \norm{h}^2 P_{n} 
        \le 
        \frac{\norm{h}^2}{\lambda_{n}} b .
    \end{align*}
    Hence, the operator $P_{n}(h \otimes h)P_{n} = (P_{n} h) \otimes (P_{n} h)$ belongs to the span of $[0,b]$. 
    Moreover, a brief computation shows that $(P_{n} h) \otimes (P_{n} h)$ converges to $h \otimes h$ with respect to the trace norm as $n \to \infty$, so the proof is complete.
\end{proof}

A result similar to Lemma~\ref{lem:trace-class-order-intervals}(ii) also holds in the space $\Kp{H}$, see \cite[Examples~2.15(i) and~2.16]{glueck:2020}.
Now we can prove the main result of this section.

\begin{proof}[Proof of Theorem~\ref{thm:trace-class}]
    The implication (a) $\implies$ (b) is a special case of Proposition~\ref{prop:isometry-unital} and (b) $\implies$ (c) is obvious. 

    (c) $\implies$ (a):    
    Since $T$ is doubly power bounded, we have $0 \not\in \Sp{T}$. 
    The pre-adjoint $T_{*}$ is also doubly power bounded, so it it suffices to show that its inverse $T_{*}^{-1}$ is positive. 
    
    For every $a \in [0,b]$ the positivity of $T_{*}$ and the inequality $T_{*} b \le b$ imply that the orbit $\{T_{*}^n a \colon n \ge 0\}$ is contained in $[0,b]$. 
    The order interval $[0,b]$ is norm compact in $\mathcal{L}^1(H)$ according to Lemma~\ref{lem:trace-class-order-intervals}(i), so the orbit is relatively norm compact in $\mathcal{L}^1(H)$. 
    Since $T_{*}$ is power bounded and the span of $[0,b]$ is dense in $\mathcal{L}^1(H)$ according to Lemma~\ref{lem:trace-class-order-intervals}(ii), it follows that the orbit $\{T_{*}^n a \colon n \ge 0\}$ is relatively norm compact in $\mathcal{L}^1(H)$ for every $a \in \mathcal{L}^1(H)$, see \cite[Corollary~A.5]{engel:2000}. 

    So $T_{*}$ is almost periodic and hence, the Jacobs-de Leeuw-Glicksberg decomposition from Theorem~\ref{thm:jdlg} is applicable to the operator $T_{*}$. 
    For $T_{*}$ doubly power bounded, part~(iv) of the theorem shows that $\ker P = \{0\}$; so $P = \operatorname{id}$.
    Now let $1 \le n_{k} \to \infty$ be a net of integers as in part~(iii) of the same theorem. 
    Then it follows that $T_*^{n_{k} - 1} \to T_*^{-1}$ strongly as $k \to \infty$ and hence, $T_*^{-1}$ is positive, as claimed. 
\end{proof}

\begin{remark}
    The Jacobs-de Leeuw-Glicksberg decomposition can also be used to show algebraic properties of the closed span of the eigenvectors of an operator that belong to the unimodular eigenvalues. 
    We illustrate this briefly for one concrete situation. 
    
    Let $H$ be a complex Hilbert space and let $T \colon  \Kp{H} \to \Kp{H}$ be a positive and power bounded linear operator. 
    Assume moreover that there exists an injective operator $b \in \Kp{H}$ such that $Tb = b$. 
    Since the order interval $[0,b]$ is norm compact in $\Kp{H}$ by Lemma~\ref{lem:order-interval-compact} and since the linear span of $[0,b]$ is dense in $\Kp{H}$ according to \cite[Examples~2.15(i) and~2.16]{glueck:2020}, the same argument as in the proof of Theorem~\ref{thm:trace-class} shows that $T$ is almost periodic. 
    Hence, the Jacobs-de Leeuw-Glicksberg decomposition from Theorem~\ref{thm:jdlg} can be applied and gives a positive projection $P\colon \Kp{H} \to \Kp{H}$ as stated in the theorem. 

    For self-adjoint elements $x$, $y$ in the range of $P$ one can define the binary operation $x \star y \coloneqq P(x \circ y)$, 
    which defines a Jordan product on the range of $P$ (see  \cite[Theorem~1.4]{effros-stoermer:1979}).     
    Similarly, if $T$ is $2$-positive, then so is $P$ and hence $(x,y) \mapsto P(xy)$ defines a product that turns the range of $P$ into a \CA-algebra, see \cite[Theorem~1.3.13]{blecher:2004}.
    If $P$ is faithful, then the assumption \emph{Schwarz operator} is sufficient.

    If $P$ is faithful, then in both cases the algebraic operators defined on the range of $P$ coincide with the same operations on $\BA$.
\end{remark}

\section{Positive inverses in infinite dimensions III: the doubly power bounded case on atomic von Neumann algebras}
    \label{sec:atomic}

In this last section we extend our results from Section~\ref{sec:ap-operators-on-c-star-algebras} to von Neumann algebras that are \emph{atomic}, which means that every projection in $\BA$ dominates a minimal projection \cite[Definition~III.5.9.]{takesaki:1979}. 
More on atomic von Neumann algebras can be found in \cite[IV.2.2.1--IV.2.2.4]{blackadar:2006} and \cite[Exercise V.2.8]{takesaki:1979}.

Recall that a bounded linear form $\phi$ on a von Neumann algebra $\BA$ is called \emph{normal} if $\phi \in \BA_{*}$, where $\BA_{*}$ is the predual of $\BA$.
If such a $\phi$ is positive, we denote by $\supp{\phi}$ the support projection of $\phi$, \ie $\supp{\phi} = 1 - \sup \mathcal{P}$ where 
\[
    \mathcal{P} = \{ p \in \BA_{h} \colon \text{ $p$ projection and $\phi(p)=0$} \}.
\]
We refer to \cite[5.15]{stratila:1979} for various properties of $\supp{\phi}$. 
The following result generalizes Lemma~\ref{lem:trace-class-order-intervals} to preduals of atomic von Neumann algebras. 

\begin{lemma}
    \label{lem:atomic-order-intervals}
    Let $\BA$ be an atomic von Neumann algebra and let $0 \le \phi \in \BA_*$ by a normal state.
    \begin{enumerate}[\upshape (i)]
        \item 
        If $\BA$ is atomic, then the order interval 
        \begin{align*}
            [0,\phi] \coloneqq \{\psi \in \BA_* \colon 0 \le \psi \le \varphi\}
        \end{align*}
        is norm compact in $\BA_*$.

        \item 
        If $\phi$ is faithful, then the face generated by $\phi$ in the positive cone $\BA_{*}^{+}$, \ie the set
        \[
            F_{\phi} \coloneqq \bigcup_{n} n[0,\phi],
        \]
        is dense in $\BA_{*}^{+}$. 
        Hence, the linear span of $[0, \phi]$ is dense in $\BA_*$.
    \end{enumerate}
\end{lemma}

\begin{proof}
    (i) 
    The order interval $[0,\phi]$ is weakly compact, hence every sequence $(\psi_{n})$ has a weakly convergent subsequence by Eberlein's theorem.
    But if $\psi_{n_{k}} \to \psi$ weakly, then $\lim_{k}\norm{\psi_{n_{k}} - \psi} = 0 $ by \cite[Corollary~III.5.11]{takesaki:1979}.

    (ii)
    The closure of the face $F_\varphi$ is given by
    \[
        \overline{F_{\phi}} = \{ 0 \leq \psi \in \BA_{*} \colon \supp{\psi} \leq \supp{\phi} \},
    \]
    see \cite[Lemma~4.1]{effros:1963}.
    Hence, if $\phi$ acts faithfully on $\BA$, then $\overline{F_{\phi}}$ equals  in the predual $\BA_*$ since $\supp{\phi}=\1$.
\end{proof}

Let $T$ be a positive operator on a von Neumann algebra $\BA$ and consider its adjoint operator $T'$ on the dual space of $\BA$.
If $T'\phi \leq \phi$ for a normal state $0 \le \phi \in \BA_*$, then $T'$ leaves the face $F_{\phi}$ invariant and if $\phi$ is faithful, it thus follows that $T'\BA_{*} \subseteq \BA_{*}$.
Therefore the restriction of $T'$ onto $\BA_{*} $ defines a positive operator $T_{*}$ which leaves the order interval $[0,\phi]$ invariant.

Now we can extend Theorem~\ref{thm:trace-class} to atomic von Neumann algebras.

\begin{theorem}
    \label{thm:atomic-class}
    Let $\BA$ be a von Neumann algebra and let $0 \le \phi \in \BA_{*}$ be a faithful and normal state on $\BA$. 
    Let $T \colon \BA \to \BA$ be a positive and unital linear operator such that $T'\phi \leq \phi$. 
    Then the following are equivalent.
    \begin{enumerate}[\upshape(a)]
        \item 
        The operator $T$ is invertible and its inverse operator $T^{-1}$ is positive.

        \item 
        The operator $T$ is a bijective isometry.
        
        \item 
        The operator $T$ is doubly power bounded.
    \end{enumerate}
\end{theorem}

The proof is literally the same as the proof of Theorem~\ref{thm:trace-class} -- one just has to replace the references to Lemma~\ref{lem:trace-class-order-intervals} by references to Lemma~\ref{lem:atomic-order-intervals}. 

\begin{remark}
    In the situation of Theorem~\ref{thm:atomic-class} the von Neumman algebra $\BA$ is $\sigma$-finite since $\phi$ is faithful. 
    Moreover, since $\BA$ is atomic one has 
    \[
        \BA = \bigoplus_{k\in N} \mathcal{L}(H_{k}),
    \]
    where $N$ is a countable set and $H_{k}$ is a Hilbert space for every $k\in N$ which can be infinite-dimensional (see \cite[Theorem~IV.2.2.2]{blackadar:2006}).
\end{remark}

We find it worthwhile to state the commutative case explicitly as a corollary.

\begin{corollary}
    \label{cor:commutative}
    Let $T \colon \ell^\infty \to \ell^\infty$ be a positive linear operator satisfying $T \1 = \1$. 
    Assume that there exists a vector $b \in \ell^{1}$ such that $b_{k} > 0$ for each index $k$ and such that $T'b \le b$.
    
    Then $T$ is invertible with positive inverse $T^{-1}$ if and only if $T$ is a bijective isometry if and only if $T$ is doubly power bounded. 
\end{corollary}

This follows readily if one applies Theorem~\ref{thm:atomic-class} to the atomic von Neumann algebra $\ell^\infty$.
Of course one can also prove the corollary directly, without refering to von Neumann algebra theory, by copying the proof of Theorem~\ref{thm:trace-class} and using the well-known -- and easy to prove -- fact that order intervals in $\ell^{1}$ are norm compact.

Let us stress again that we do not know the answer to Open Problem~\ref{probl:without-doubly-pb} in the commutative case, either. 
Example~\ref{exa:perturbed-shift} does not provide a counterexample since the pre-adjoint $T_{*}$ on $\ell^{1}$ does not have a strictly positive fixed vector 
(and if it had one, then the double power boundedness of $T$ would imply that $T^{-1}$ is positive, as we have just pointed out).

\bibliographystyle{plain}
\bibliography{automorphism}

\end{document}